\documentclass[leqno]{siamltex}

\usepackage{calc}

\usepackage{graphicx}
\usepackage{amssymb}
\usepackage{amsmath}
\usepackage{tikz}
\usetikzlibrary{arrows}
\usepackage{amsopn}
\usepackage{epsfig}
\usepackage{enumerate}
\usepackage{multirow}
\usepackage{subcaption}
\usepackage{amsrefs}
\usepackage{float}

\numberwithin{figure}{section}
\numberwithin{table}{section}

\tikzstyle{decision} = [rectangle, draw, fill=blue!20, 
    text width=10em, text badly centered]
\tikzstyle{block} = [rectangle, draw, fill=magenta!20, 
    text width=10em, text centered, rounded corners, minimum height=4em]
\tikzstyle{block2} = [rectangle, draw, fill=blue!20, 
    text width=10em, text centered, rounded corners, minimum height=4em]
\tikzstyle{line} = [draw, -latex]
\tikzstyle{cloud} = [draw, ellipse,fill=red!20, node distance=1cm,
    minimum height=2em]
\tikzstyle{empty} = [node distance=4cm]

\everymath{\displaystyle}

\title{Analysis and Simulation of the Three-Component Model of HIV Dynamics}

\author{Eric Jones\footnotemark[2] \and Peter Roemer\footnotemark[1]\\ Faculty Advisors: Mrinal Raghupathi\footnotemark[1] \and Stephen Pankavich\footnotemark[2]}

\begin{document}
\maketitle

\renewcommand{\thefootnote}{\fnsymbol{footnote}}
\footnotetext[1]{Department of Mathematics United States Naval Academy, 572C Holloway Road, Chauvenet Hall, Annapolis, MD 21402 ({\tt peteusna@gmail.com}, {\tt raghupat@usna.edu}).}

\footnotetext[2]{Department of Applied Mathematics and Statistics, Colorado School of Mines, Golden, Colorado 80401 ({\tt erijones@mymail.mines.edu}, {\tt pankavic@mines.edu}).}

\renewcommand{\thefootnote}{\arabic{footnote}}

\begin{abstract}
  Mathematical modeling of biological systems is crucial to
  effectively and efficiently developing treatments for medical
  conditions that plague humanity. Often, systems of ordinary
  differential equations are a traditional tool used to describe the
  spread of disease within the body.  We consider the dynamics of the
  Human Immunodeficiency Virus (HIV) in vivo during the initial stages
  of infection. In particular, we examine the well-known
  three-component model and prove the existence, uniqueness,
  and boundedness of solutions. Furthermore, we prove that solutions
  remain biologically meaningful, i.e., are positivity preserving, and
  perform a thorough, local stability analysis for the equilibrium
  states of the system. Finally, we incorporate random coefficients
  within the model and obtain numerical results to predict the
  probability of infection given the transmission of the virus to a
  new individual.
  \end{abstract}



\pagestyle{myheadings}
\thispagestyle{plain}
\markboth{E. JONES AND P. ROEMER}{ANALYSIS AND SIMULATION OF 3CM}

%\author[Pankavich]{Stephen Pankavich}
%\address[Stephen Pankavich]{Department of Applied Mathematics and Statistics, Colorado School of Mines, Golden, CO 80401}
%\email{pankavic@mines.edu}
%
%
%
%\author[Raghupathi]{Mrinal Raghupathi}
%
%\address[Mrinal Raghupathi and Peter Roemer]{Department of Mathematics, United States Naval Academy, Annapolis, MD 21402}
%
%\email{raghupat@usna.edu}
%
%\author[Roemer]{Peter Roemer}
%\email{m135910@usna.edu}

\section{Introduction}
The human body has been studied for hundreds of years, and the
knowledge accumulated to date includes intricate molecular level
mechanisms that determine the ways in which a virus may infect cells
and replicate.  This information is necessary to develop treatments
for maladies ranging from Hepatitis to the Human Immunodeficiency
Virus (HIV) and Acquired Immunodeficiency Syndrome (AIDS).  An
understanding of the accumulation of millions of these interactions
across a timescale of medical relevance permits the creation of better
treatment methods.  As recently as fifty years ago, this could only be
done via experiment and repeated trials.  Today, through the advent of
high speed computing, viral kinetics can be simulated using
mathematical models.  Hence, one can utilize our current understanding
of molecular and cellular dynamics, describe them in mathematical
terms, and simulate interactions within the body to examine the course
of a virus from initial transmission to a long-term infection.  In
this paper we will consider the three-component model (3CM) for
HIV~\cite{BCN, KW, NB, DS, NM}. We will study the infection within a single human host
using this system of differential equations, analyze the resulting
dynamics of the model, and simulate the behavior of solutions.

Generally speaking, HIV infection without the inclusion of
anti-retroviral therapy (ART) is described by a number of distinct
phases~\cite{PerelsonSIAM}. In the early stages of HIV infection,
symptomatic primary infection yields high concentrations of virions
within an individual's blood or tissue.  After several weeks, the
flu-like symptoms disappear and the viral density then declines
rapidly within several days.  This corresponds to an increase in the
amount of cytotoxic T lymphocytes and the subsequent appearance of
anti-HIV antibodies in the blood.  For years afterward, the viral
concentration deviates very little from this low level and the host
typically does not exhibit symptoms of HIV infection, but the
concentration of CD$4^+$ T cells measured in blood slowly
declines. Such a period can last as long as $10$ years. Ultimately,
the viral load increases, the T-cell count drops below $200$ cells per
$\mu L$ of blood, and the symptoms of AIDS appear.

Within the current study, we will focus on the initial stages of
transmission and infection, detailing the quick rise of the virion
population within the body and the noted decrease in the number of
CD$4^+$ T-cells.  Mathematical models for HIV population dynamics have
been used to understand the basic mechanisms involved in the evolution
of the infection at the microscopic level and to ascertain the effects
of anti-retroviral therapy \cite{PerelsonSIAM, PKdB}.  To date,
deterministic models of HIV dynamics have included at least two
components: the population of uninfected CD$4^+$ T-cells and the
density of virus-producing cells.  However, in the construction of a
model for early HIV dynamics one may consider several components as
populations within the the blood or lymphatic tissue, including
latently-infected CD$4^+$ T-cells, macrophages, and actively-producing
infected macrophages.  As previously described, the time-dependent
changes in the cytotoxic T-lymphocytes which directly attack virus
producing CD$4^+$ T-cells could also be included within a model, but
as previously stated, this may not have significant impact within the
earliest stage of infection.  Similarly, effects of time delays
between infection of T-cells and active production of virions, or
multiple compartment models, which couple the dynamics of multiple
infected regions of the body, may also play a role in later stages,
but are omitted from the current study.  Nevertheless, some authors
have found it useful to consider simple models which include only the
density of virions, the uninfected CD$4^+$ T-cell population, and one
class of (actively-productive) infected CD$4^+$ T-cells.  This gives
rise to three-component models as in \cite{BCN, KW, NB}.  We emphasize
that the current article focuses only on the early period, up to a
couple of months after infection, and does not study the later
progression to the acquired immune deficiency syndrome (AIDS) which
may follow without the use of anti-retroviral therapy.

\section{Background and Derivation of Mathematical Model}
 
The general principle behind any sort of mathematical model is to
examine closely the interactions between the quantities being
analyzed. We begin with a small number of initial axioms, and then
follow the implications thoroughly to their conclusion. As a basis for
our model, we utilize a general biological understanding of HIV
dynamics, including infection, replication, and clearance.

In this section we describe the three-component model that has been
widely used in the study of HIV.  There are many biological operators
involved in the interaction between HIV and cells within the human
body. The first group is a subset of the population of lymphocytes,
which in turn are a type of white blood cell. This subset is known as
CD$4^+$ T-cells, or helper T-cells. These T-cells have a variety of
functions, including secreting substances that stimulate the immune
system, in addition to acting as memory agents and regulating the
immune response.  In short, CD4+ T-cells detect and direct immune
system responses to invading bacteria and viruses. Without them, the
body significantly suffers from opportunistic infections that are
greater in severity and duration than they would be if the CD4+
T-cells were otherwise not present. HIV, which refers in this case to
the virus, not the disease or symptoms associated with it, is a
retrovirus that infects helper T-cells. The virus, which is
significantly smaller than the T-cell ($120$ nm in diameter
compared with $7$ $\mu$m in diameter, respectively), breaches the cell wall and
transports its RNA into the T-cell nucleus, where it may remain
dormant for a time. Upon activation, the T-cell ceases its function as
part of the immune system and instead produces additional copies of
HIV. These infected T-cells, along with the healthy T-cells and free
floating HIV, are the populations with which we are concerned, and
will appear within the mathematical model.
 
The second aspect of the model's creation involves incorporating the
interactions between these populations and deriving their
corresponding mathematical representation.  Healthy T-cells are created from stem cells in
the bone marrow, and mature in the thymus. While production of T-cells
does decrease with the aging of the human body, we shall consider it
to be a constant process for two reasons. First, there is no known
method other than this production that can affect T-cell creation, and
second, the time scale of interest within the model is sufficiently
small to consider the T-cell production rate as constant. When
considering removal of these cells, we note that T-cells do age and,
in time, expire. Within the model, we assume that each T-cell
functions for roughly the same amount of time, and thus, the death
rate does not vary over the entire population. Instead, we assume the
overall number of T-cells lost in a group over a certain period of
time is proportional to the number of T-cells within the group. The
other mechanism through which the population of healthy T-cells may
decrease is via infection.  In considering such effects, we utilize an
interaction term that arises from the widely-used ``mass action
principle'' to describe the transfer between populations when the
virus infects healthy T-cells. This mass action term represents the
idea that the rate of interaction, or infection, is directly
proportional to the product of the participating populations, namely
those of virions (or virus particles) and healthy T-cells.
This completes the interactions for the healthy T-cells. We note that
the only way to increase the population of infected T-cells is through
HIV infection. Thus, our model will contain the same mass action term
to describe the removal of the healthy T-cells and the addition of
infected T-cells. Similar to the healthy T-cells, infected T-cells die-off or are cleared by the immune system at a rate proportional to the
size of their current population. The virus, while produced from the
infected T-cells, does not cause the destruction of infected T-cells.
Thus, such a transition is not included within the model.
 
While the virus production rate does differ from cell to cell, we can
assume that the aggregate rate is proportional to the
population of infected T-cells. This is the only mechanism by which
the virus can be created. In contrast, there are two manners in
which the virus can be removed, or cleared, from the body. The first
is through viral infection of T-cells. The act of infecting a healthy
T-cell must technically remove a virus particle from the population of
viruses that can infect further T-cells. However, when considering the
overall population quantities, the amount of viruses lost this way is
minute compared to other methods of creation and destruction. Hence,
we will omit this mechanism.  The second method of removal is known as
viral clearance. It is the removal by the body of individual virus
particles, and is performed at a rate proportional to the current
amount of virus particles within the body.

% For all of the interactions we have described, the rates of change
% were proportional to a population. In the equations that follow, we
% endow each of these interactions with a proportionality constant, in
% the necessary units to guarantee that each term within the equation is
% measured in organisms (cells or virions) per unit time.  For instance,
% the creation rate of the healthy T-cells will be represented by the
% constant $\lambda >0$, whose unit of measurement is T-cells per unit
% time. Analogously, the rate coefficient for the creation of infected
% T-cells is measured in units of time per virus particle. Thus, when we
% sum the interactions for each population, we arrive at a growth
% rate. Mathematically, the rate of change of a quantity is given by its
% derivative, and hence a differential equation results.
%In our case, the system is known as the Three Component Model (3CM) \cite{BCN, KW, NB, TuckShip}.

We denote by $T$, $I$, and $V$, the number of healthy T-cells, infected T-cells and virions respectively. Based on the biological description above we have the following system of three ODEs:

\begin{equation}
\tag{3CM}
\label{3CM}
\left \{ \begin{aligned}
\frac{dT} {dt} &= \lambda -  \mu T - k T V \\
\frac{dI} {dt} &=  k T V - \delta I \\
\frac {dV} {dt} &= p I - c V.
\end{aligned} \right.
\end{equation}

The parameters $\lambda, \mu, k, \delta, p, c$ play an important role
in our later results on viral persistence. Based on biological
considerations we assume that these constants are positive.
Table~\ref{tab1} shows typical values for these constants and the corresponding units. The values are the average values of these parameters from \cite{SCCDHP}.

\begin{table}[ht] 
\vspace{0.1in}
\begin{center}   
\begin{tabular}{|c||c|c|c|c|c|}
	\hline
	Parameter & Biological Process & Minimum & Mean Value & Maximum & Units \\
	\hline
	\hline
	$\lambda$ & T-cell growth rate & $0.043$ & $0.1089$ & $0.2$ & $\mu L^{-1}$ day$^{-1}$ \\
	\hline
	$\mu$ & T-cell death rate & $0.0043$ & $0.01089$ & $0.02$ & day$^{-1}$ \\
	\hline
	$k$ & Infection rate & $1.9 \times 10^{-4}$ & $1.179 \times 10^{-3}$ & $4.8 \times 10^{-3}$ & $\mu L$ day$^{-1}$\\
	\hline
	$\delta$ & Infected T-cell death rate & $0.13$ & $0.3660$ & $0.8$ & day$^{-1}$ \\
	\hline
	$p$ & Virus production rate & $98$ & $1.427 \times 10^3$ & $7.1 \times 10^3$ & day$^{-1}$ \\
	\hline
	$c$ & Viral clearance rate & $3$ & $3$ & $3$ & day$^{-1}$ \\
	\hline
\end{tabular}

\caption{Parameter values for (\ref{3CM}) as observed in \cite{SCCDHP}}
\label{tab1} 
\end{center}
\end{table}
%\vspace{0.1in}

% Combining the multiple interactions that govern the change in
% population of the healthy T-cells, including their natural growth
% rate, natural death rate, and infection rate, we arrive at the
% differential equation for their rate of change
% \begin{equation}
% \frac{dT}{dt}=\lambda-\mu T -kTV.
% \end{equation}
% %
% This equation mathematically describes the biological interactions
% presented above. A similar equation is derived for the infected
% T-cells, namely
% %
% \begin{equation}
% \frac{dI}{dt} = kTV - \delta I.
% \end{equation}
% %
% Here, it is important to note that the $kTV$ terms in the first and
% second equations represent the same interaction and yield the same
% number of cells transferred from the healthy T-cell population to the
% infected T-cell population. Next, we consider the rate of change of
% the virus population and derive the corresponding differential
% equation to represent its dynamics
% \begin{equation}
% \frac{dV}{dt}=pI-cV.
% \end{equation}

The diagram (Fig \ref{tikz3CM}) provides a brief visual representation of
the mechanisms which govern the system of differential equations. We
see each of the populations within a circle, while the arrows running
to and from each circle describe their respective interaction.

\tikzstyle{line} = [draw, -latex']

\begin{figure}
\begin{center}
\begin{tikzpicture}[place/.style={circle,draw=blue!50,line width=0.4mm, fill=white},
   transition/.style={->,line width=0.4mm}, scale = 0.5]
\node[place] (t) at (0, 0) {$T$};

\node[place](i) at (3, -3) {$I$};

\node[place](v) at (-3, -3) {$V$};

\node[](vpathstart) at (-6, 0) {};

\node[](vpathend)at (0, -6) {};

\node[](tpathstart) at (-3, 3){};

\node[](tpathend) at (6, -6) {};

\node[](tpathend2) at (3, 3) {};

\path
	(tpathstart) edge[transition] node [auto]{{\Large $\lambda$}}(t)
	(t) edge[transition] node[auto] {\Large{$kTV$}}(i)
	(i) edge[transition] node [auto] {\Large{$\delta I$}} (tpathend)
	(vpathstart) edge[transition] node [auto] {\Large{$pI$}} (v)
	(v) edge[transition] node [auto] {\Large{$cV$}} (vpathend)
	(t) edge[transition] node [auto] {\Large{$\mu T$}} (tpathend2);

\end{tikzpicture}
\caption{Illustration of 3CM}
\label{tikz3CM}
\end{center}
\end{figure}
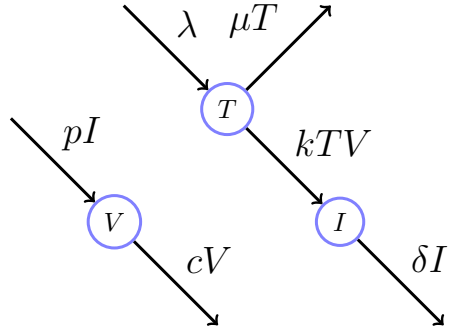

%\tikzstyle{line} = [draw, -latex']
%\begin{center}
%\begin{figure}
%\label{tikz3CM}
%\begin{tikzpicture}
%  [place/.style={circle,draw=blue!50,line width=1mm, fill=gray!10},
%   transition/.style={->,line width=1mm}]
%\node[place, minimum size=2cm] (t) at (0, 0) {$T$};
%
%\node[place,minimum size=2cm](i) at (3, -3) {$I$};
%
%\node[place,minimum size=2cm](v) at (-3, -3) {$V$};
%
%\node[](vpathstart) at (-6, 0) {};
%
%\node[](vpathend)at (0, -6) {};
%
%\node[](tpathstart) at (-3, 3){};
%
%\node[](tpathend) at (6, -6) {};
%
%\node[](tpathend2) at (3, 3) {};
%
%\path
%	(tpathstart) edge[transition] node [auto]{{\Large $\lambda$}}(t)
%	(t) edge[transition] node[auto] {\Large{$kTV$}}(i)
%	(i) edge[transition] node [auto] {\Large{$\delta I$}} (tpathend)
%	(vpathstart) edge[transition] node [auto] {\Large{$pI$}} (v)
%	(v) edge[transition] node [auto] {\Large{$cV$}} (vpathend)
%	(t) edge[transition] node [auto] {\Large{$\mu T$}} (tpathend2);
%	    
%
%\end{tikzpicture}
%\caption{Illustration of 3CM}
%\end{figure}
%\end{center}

% Clearly, the unknown populations depend upon one another, and hence they are not examined independently. Just as the interaction of populations within the body serves to connect the behavior of virions and T-cells, the functions within the differential equations above are also coupled. They describe, for instance, the notion that an increase within the virus population will cause a similar increase within the infected T-cell population. Because of these interactions, we study the coupled system of equations known together as the three-component model \cite{}

%%%%%%%%%%%%%%%%%%%%%%%%%%%%%%%
\section{Mathematical Analysis}
We now answer some fundamental questions about 3CM. In particular we
will show that solutions to 3CM exist for all positive time, and are
unique. Later we will show that the solution converge to one of two
possible steady-states and that the solutions to 3CM remain positive
given positive initial conditions. This last property is important
since it shows that the model is biologically relevant.

\subsection{Properties of Solutions}

The first step in examining (\ref{3CM}) is to prove that a solution to
the initial-value problem does, in fact, exist, and that this solution
is unique.

\begin{theorem}
\label{T1}
Let $T_0,I_0, V_0 \in \mathbb{R}$ be given. There exists $t_0 > 0$ and
continuously differentiable functions $T, I, V : [0,t_0) \to
\mathbb{R}$ such that the ordered triple $(T,I,V)$ satisfies
(\ref{3CM}) and $(T,I,V)(0) = (T_0,I_0,V_0)$.
\end{theorem}

\begin{proof}
  To prove the result, we utilize the classical Picard-Lindel\"{o}f
  Theorem (cf. \cite{Bartle}).  Since the system of ODEs is
  autonomous, it suffices to show that the function $f : \mathbb{R}^3
  \to \mathbb{R}^3$ defined by $$f(y) = \begin{bmatrix}
    \lambda - \mu y_1 - ky_1y_3 \\
    ky_1y_3 - \delta y_2 \\
    py_2 - c y_3
\end{bmatrix}$$
is locally Lipschitz in its $y$ argument.
In fact, it is enough to notice that the Jacobian matrix $$\nabla f(y) = \begin{bmatrix}
-\mu - ky_3 & 0 & -ky_1 \\
k y_3 & - \delta & ky_1  \\
0 & p & - c
\end{bmatrix}$$ is linear in $y$ and therefore locally bounded for
every $y \in \mathbb{R}^3$.  Hence, $f$ has a continuous, bounded
derivative on any compact subset of $\mathbb{R}^3$ and so $f$ is
locally Lipschitz in $y$. By the Picard- Lindel\"{o}f Theorem, there
exists a unique solution, $y(t)$, to the ordinary differential
equation $y'(t)=f(y(t))$ on $[0,t_0]$ for some time $t_0 > 0$.
\end{proof}

Additionally, we may show that for positive initial data, solutions remain positive as long as they exist.  A fortunate byproduct of this result is that the solutions are also bounded.

\begin{theorem}[Boundedness and Positivity]
\label{T2}
Assume the initial conditions of (\ref{3CM}) satisfy $T_0>0$,
$I_0>0$, and $V_0>0$.  If the unique solution provided by Theorem
\ref{T1} exists on the interval $[0,t_0]$ for some $t_0>0$, then the
functions $T(t), I(t),$ and $V(t)$ will be bounded and remain positive
for all $t \in [0,t_0]$.

\end{theorem}
\begin{proof}
  We assume that $T(t)$, $I(t)$, and $V(t)$ initially have positive
  values. From the previous theorem, there exists a $t^*$ such that
  the solution exists on $[0,t^*]$. Let us denote by $T^*$ the largest
  time for which all populations remain positive, or more precisely
$$ T^* = \sup\{t \in [0,t^*] : T(s), I(s), V(s) > 0 \ \mbox{for all} \ s \in [0,t]\}.$$
Then on the interval $[0,T^*]$ we can make estimate the population values.

Recall that all constants in the system are positive.  Using
this and the positivity of solutions on $[0,T^*]$, we can place lower bounds on $\frac{dI}{dt}$
and $\frac{dV}{dt}$ since 
$$\frac{dI}{dt}=kTV-\delta I \geq -\delta I$$
and 
$$\frac{dV}{dt}=pI-cV \geq -cV.$$
Using an integrating factor, we rewrite these differential inequalities to find
$$ I(t) \geq I(0)e^{-\delta t} > 0$$
and  
$$V(t) \geq V(0)e^{- c t} >0$$
for $t \in [0,T^*]$.  Similarly, we can place an upper bound on
$\frac{dT}{dt}$ so that 
$$\frac{dT}{dt}=\lambda-\mu T -kTV \leq \lambda.$$
Solving for $T$ yields
$$T(t) \leq T(0) + \lambda t \leq C_1(1+t).$$
where the constant $C_1$ satisfies $C_1\geq \max\{\lambda, T(0)\}$.
We can sum the equations for $\frac{dI}{dt}$ and $\frac{dV}{dt}$ and place bounds on this sum so that
$$\frac{d}{dt}(I+V)=kTV-\delta I +  pI - cV  \leq kTV+pI.$$ 
Recall that we have a bound on $T$, so we can substitute
$$\frac{d}{dt}(I+V)  \leq k C_1(1+t)V+pI \leq C_2(1+t)(I+V)$$
where $C_2 \geq \max\{kC_1, p\}$.
Solving the differential equation yields 
$$ (I+V)(t) \leq C_3e^{t^2}$$
for $ t\in[0,T^*]$  where $C_3 > 0$ depends upon $C_2$, $I(0)$, and $V(0)$ only.
Since $I(t)$ is positive, we can place an upper bound on $V$ by 
$$C_3e^{t^2} \geq (I+V)(t) \geq V(t)$$
Additionally, since $V(t)$ is positive, it follows that $I(t)$ must be as well.
$$C_3e^{t^2} \geq (I+V)(t) \geq I(t)$$
With these bounds in place, we can now examine $T(t)$ and bound it from below using
\begin{eqnarray*}
\frac{dT}{dt} & = & \lambda-\mu T -kTV \geq - \mu T - kTV \geq -\mu T - kC_3e^{t^2}T\\
& \geq & -C_4(1+e^{t^2})T
\end{eqnarray*}
for $t \in[0,T^*]$, where $C_4 \geq \max\{\mu, kC_3\}$.
Shifting that last term to the other side of the equation yields
$$\frac{dT}{dt} + C_4(1+e^{t^2})T \geq 0$$
Since we know
$$\frac{d}{dt} ( T(t) + e^{C_4\int_0^t(1+e^{\tau^2}d\tau)} ) \geq 0,$$
then we find for $t \in [0,T^*]$
\begin{equation}
T(t) \geq T(0) e^{-C_4\int_0^t(1+e^{\tau^2}d\tau)} > 0.
\end{equation}

Thus, the values of $T$, $I$, and $V$ stay strictly positive for all
of $[0, T^*]$, including at time $T^*$. By continuity, there must
exist a $t>T^*$ such that $T(t)$, $I(t)$, and $V(t)$ are still
positive. This contradicts the definition of $T^*$, and shows that $T(t)$, $I(t)$, and
$V(t)$ are strictly positive on the entire interval $[0, t^*]$.
Additionally, on this same interval, all of the functions remain
bounded, so the interval of existence can be extended further.  In
fact, the bounds on $T$, $I$, and $V$ derived above hold on any
compact time interval.  Thus, we may extend the time interval on which
the solution exists to $[0,t_0]$ for any $t_0 > 0$ and from the above
argument, the solutions remain both bounded and positive on $[0,t_0]$.
\end{proof}

With this, we have a general idea that the model is sound, and can say
with certainty that it remains biologically valid as long as it began
with biologically-reasonable (i.e, positive) data. This also shows
that once infected, it is entirely possible that the virus may
continue to exist beneath a detectable threshold without doing any
damage.  Finally, we remark that the bounds obtained above ensure the
global existence of solutions.

\begin{corollary}
  Let $T_0,I_0, V_0 > 0$ be given. Then, for any $t_0 > 0$
  there exist continuously differentiable functions $T, I, V : [0,t_0]
  \to \mathbb{R}$ such that the ordered triple $(T,I,V)$ satisfies
  (\ref{3CM}) and $(T,I,V)(0) = (T_0,I_0,V_0)$.
\end{corollary}

Thus, given positive initial data and any $t_0 > 0$, we can be certain that the solution stays both positive and bounded on the interval $[0,t_0]$.

\subsection{Steady States}

In order to fully understand the dynamics of the three component
model, it is necessary to first determine values of equilibria.  An
equilibrium point is a constant solution of (\ref{3CM}) so that if the
system begins at such a value, it will remain there for all time. In
other words, the populations are unchanging so the rate of change for
each population is zero.  Setting $\frac {dT} {dt}$,$\frac {dI} {dt}$,
and $\frac {dV} {dt}$ equal to zero and solving the resulting
equations for $T$, $I$, and $V$, we find that there exist exactly two
equilibria. From a biological perspective, we can categorize these
points to be when the HIV virus is either extinct from the body, i.e.,
$I=V=0$, or when the virus persists within the body $(I\neq 0, V \neq
0)$ as $t$ grows large.

We begin by solving for the nonlinear term in (\ref{3CM}) and find
$kTV = \delta I$. Additionally, the final equation implies $I =
\frac{c}{p} V$.  Using the latter equation to substitute for $I$, we
find
$$ kTV = \frac{\delta c}{p} V$$ or $$V \left ( kT - \frac{\delta c}{p}  \right ) = 0.$$  Thus, either $V = 0$ or $T = \frac{\delta c}{kp}$.  In the former case, we find $I = 0$ and thus $T = \frac{\lambda}{\mu}$.  
Hence, the ordered triplet $$(T, I, V) = \left (\frac {\lambda} {\mu}
  , 0 , 0 \right )$$ is an equilibrium solution known as \textbf{viral
  extinction}, since there are no virus particles or infected
cells. We will refer to this point as $P_E$.

In the latter case, $T= \frac {c \delta}{pk}$, and substituting this
value of $T$ into the first equation yields $V =\frac {p \lambda} {c
  \delta} - \frac {\mu} {k}$ and further substitution shows $I = \frac
{\lambda} {\delta} - \frac {\mu c}{kp}$. Thus, a second equilibrium
exists at the point $$(T, I, V) = \left (\frac {c \delta} {pk} , \frac
  {\lambda} {\delta} - \frac {\mu c} {kp} , \frac {p \lambda} {c
    \delta} - \frac {\mu} {k} \right ).$$ Since there are distinct
presences of virus particles and infected T-cells, we refer to this
point as \textbf{viral persistence} and abbreviate the point as $P_P$.

In terms of biology, we can say $P_E$ is the case in which an
infection exists for a short period of time, then is removed from the
body by natural means. The virus does not persist. The second case,
where the system of equations tends to $P_P$, denotes that situation
where the body is unable to clear the infection by itself. If this
ends up being the case, than after a certain period of time, the Three
Component Model loses its applicability as the infection takes a
deeper hold on the body. More complex models, which consider latent
infection, effects of macrophages, or cytotoxic immune response, are
then required to describe the spread of HIV within the body and its
development towards AIDS.

\subsection{Stability Analysis} 

For linear ODEs, it is well-known that the stability properties depend
only upon the eigenvalues of the system.  However, our model
(\ref{3CM}) is nonlinear, and thus we must rely on linearization and a
theorem of Hartman \& Grobman \cite{HartmanGrobman} to unify the local
behavior of the linear and nonlinear systems.

We will investigate the local stability properties of these equilibria
by approximating the nonlinear system of differential equations
(\ref{3CM}) with a linear system at the points $P_E$ and $P_P$. Then,
we locally perturb the system from equilibrium and examine the
resulting long time behavior.  This is done by linearizing the system
about each equilibria, using the Jacobian for (\ref{3CM})
$$
J_{3CM} = \begin{bmatrix}
	-kV-\mu & 0 & -kT \\
	kV & - \delta & kT \\
	0 & p & -c.
\end{bmatrix}
$$
Then, by studying the linearized system $$\dot{z}(t) = J_{3CM}(P)
z(t)$$ we can investigate the stability of each equilibrium point $P=
P_E$ and $P=P_P$.  As we will see below, this property depends only on
a single number, referred to as the basic reproduction number, $R$
given by
\begin{equation}
\label{R}
R = \frac{kp\lambda}{c\delta\mu}.
\end{equation}

We now prove two theorems that demonstrate the relationship between
the value of $R$ and the local asymptotically stability of
equilibria. These results imply that one can simply examine the value
of $R$ to determine whether viral persistence or viral extinction
occurs in the limit of the system as $t \to \infty$. This is a
remarkable result that allows for the estimation of the persistence of
HIV upon initial infection solely by Monte-Carlo simulations by
generating different values of $R$, as in \cite{TuckShip}.
\begin{theorem}
\label{ExtinctStability}
The viral extinction equilibrium $P_E$ given by $$(T, I, V) = \left (\frac {\lambda} {\mu} , 0 , 0 \right )$$
is locally asymptotically stable if and only if $R \leq 1$. 
\end{theorem}
\begin{proof}
  We begin by computing $J_{3CM}(P_E)$ and determining its
  corresponding eigenvalues, since these values are known to
  characterize the local asymptotic behavior of the associated linear
  system. Specifically, if every eigenvalue possesses negative real
  part, then the equilibrium point will be stable. On the other hand,
  if one or more of the eigenvalues possess positive real part,
  then small perturbations from equilibrium result in magnifications
  of those disturbances, and the unstable manifold is nontrivial.  
  We remark that in the rare event that $R =1$, the equilibria $P_E$ and $P_P$ are identical. 
  Hence, we'll focus on the case in which $R < 1$ since the asymptotic stability can been shown 
  when $R=1$ by using a Lyapunov function \cite{Korob} instead of an analysis of the linearized
  system.

Evaluating the Jacobian at $P_E = \left (\frac{\lambda}{\mu} , 0 , 0 \right )$ results in 
\begin{equation}
J_{3CM}(P_E)  = \begin{bmatrix}
	- \mu & 0 & - \frac {k \lambda} {\mu} \\[4 pt]
	0 & - \delta & \frac {k \lambda} {\mu}  \\[4 pt]
	0 & p & -c \\
\end{bmatrix}
\nonumber
\end{equation}
The corresponding characteristic equation can be written as
$$ \det \left [x\mathbb{I} - J_{3CM}(P_E)\right ] = 0$$ or 
$$ (x+\mu)\left ((x+\delta)(x+c)-\frac{kp\lambda}{\mu} \right) = 0.$$
Thus, $x = - \mu < 0$ is one negative eigenvalue of the system, The
remaining quadratic equation is
$$ x^2+a_1x+a_2  = 0$$
where $ a_1 = c+\delta$ and $a_2 = c\delta - \frac{kp\lambda}{\mu}$.  Thus, the other eigenvalues are
$$ x_\pm = \frac{-(c + \delta) \pm \sqrt{(c+\delta)^2 - 4(c\delta - \frac{kp\lambda}{\mu}})}{2}. $$
Since the first term under the square root is nonnegative, these eigenvalues have negative real part if and only if 
$$4\left (c\delta - \frac{kp\lambda}{\mu} \right ) > 0$$ or $$ 4c \delta(1 - R) > 0.$$  Since all parameters are positive, we see that all eigenvalues
possess negative real part if and only if $R < 1$.  Thus, in this case the origin is a locally asymptotically stable equilibrium for the system
$$\dot{z}(t) = J_{3CM}(P_E) z(t).$$
Finally, by the Hartman-Grobman Theorem, the asymptotic behavior of (\ref{3CM}) is equivalent to that of this
linear system for local perturbations, and the result follows. For a more detailed look at the transference of stability properties from linear to nonlinear systems of ordinary differential equations, see \cite{Logan}. 
\end{proof}

Now that $R$ has been incorporated within the stability analysis, we can rewrite the viral persistence equilibrium in terms of $R$, and notice that it possesses nonpositive population values for $R \leq 1$.  Hence, it should not be surprising that solutions do not tend to this equilibrium for such values of $R$.  However, as long as $R > 1$, we find that viral persistence is stable.

\begin{theorem}
\label{PersistStability}
The viral persistence equilibrium $P_P$ given by $$(T, I, V) = \left (\frac {\lambda}{\mu R}, \frac {\lambda}{\delta R}( R - 1), \frac {\mu} {k} (R- 1) \right )$$
is locally asymptotically stable if and only if $R > 1$. 
\end{theorem}

\begin{proof}
The analysis for $P_P$ is similar to that of $P_E$. We first linearize (\ref{3CM}) about $P_P$ and examine the characteristic equation. The Jacobian is slightly more complicated in this case, but it is given by
$$ 
\begin{gathered}
J_{3CM}(P_P)  =
\begin{bmatrix}
	-k\left (\frac {p \lambda} {c \delta} - \frac {\mu} {k} \right )-\mu && 0 && -k \left (\frac {c \delta} {pk} \right ) \\[4 pt]
	k \left (\frac {p \lambda} {c \delta} - \frac {\mu} {k} \right ) && - \delta && k \left(\frac {c \delta} {pk} \right )  \\[4 pt]
	0 && p && -c \\
\end{bmatrix}\\
= 
\begin{bmatrix}
	-\frac{k\lambda p}{c\delta} && 0 && -\frac{c\delta}{p} \\[4 pt]
	\frac{k\lambda p}{c\delta}-\mu && -\delta && \frac{c\delta}{p} \\[4 pt]
	0 && p && -c.
\end{bmatrix}
\end{gathered}
$$
This results in the characteristic equation
\begin{equation}
\left (x + \frac{k\lambda p}{c\delta} \right )\biggl ((x+\delta)(x+c)-c\delta)\biggr )+\frac{c\delta}{p} \left ({\frac{k\lambda p}{c\delta}-\mu} \right )p = 0
\nonumber
\end{equation}
with expanded form
$$x^3+a_1x^2+a_2x+a_3  = 0  $$
where 
$$\begin{gathered}
a_1 = c+\delta +\frac{k\lambda p}{c\delta}\\
a_2 =  \frac{k \lambda p}{\delta} +  \frac{k \lambda p}{c}\\
a_3 = k \lambda p - c \delta \mu.
\end{gathered} $$

In the case of (\ref{3CM}), it is possible to determine the signs of the solutions to this equation using a theorem of Routh and Hurwitz \cite{Routh, Hurwitz}.
According to the Routh-Hurwitz criteria, all roots of this cubic equation possess negative real part if and only if $a_1, a_2, a_3>0$ and $a_1a_2> a_3$. Hence, it is sufficient to show that $R>1$ if and only if the Routh-Hurwitz criteria are satisfied.

Let us first assume $R > 1$.  Then, $a_3= k \lambda p - c \delta \mu = c \delta \mu (R-1) > 0$ and since all the coefficients in the system are positive, $a_1, a_2 > 0$. Additionally,  $$a_1a_2=\left (c+\delta +\frac{k\lambda p}{c\delta} \right ) \left (\frac{k \lambda p}{\delta} +\frac{k \lambda p}{c} \right )>c \cdot \frac{k \lambda p}{c} = k\lambda p>k\lambda p -c\delta\mu=a_3 .$$
Thus, $R>1$ implies that $P_P$ is a locally asymptotically stable equilibrium. The other direction follows trivially since $a_3 > 0$ implies $R > 1$. Hence, this is both a necessary and sufficient condition due to the form of $a_3$, and the proof is complete.

\end{proof}

\begin{figure}[t]
\centering 
\includegraphics[scale = 0.65]{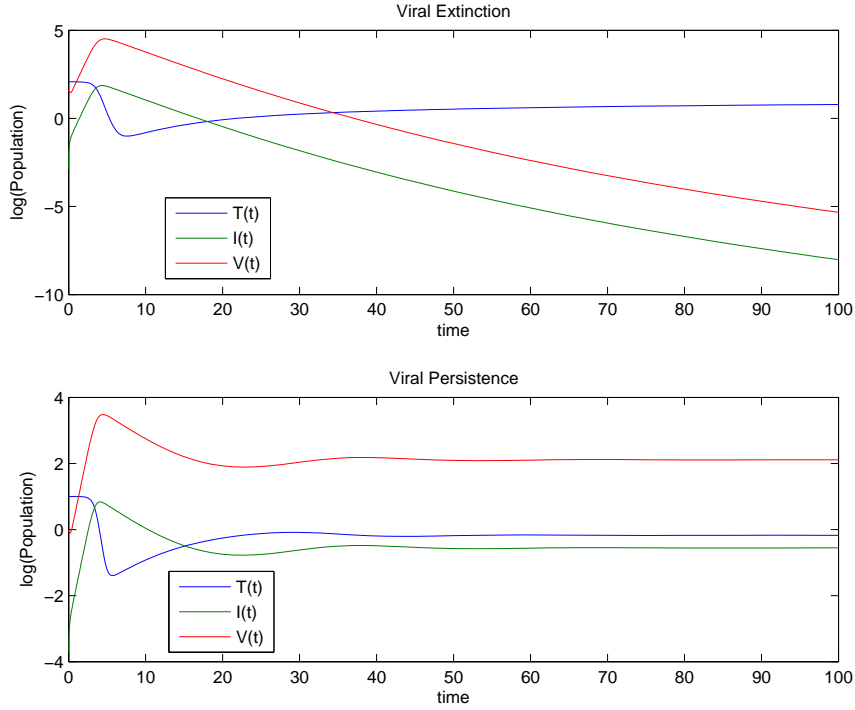}
\caption[Graphs of $R<1$, $R>1$]{ \footnotesize Graphs of solutions to  (\ref{3CM}) with illustrative coefficients. Parameter values yield either $R \leq 1$ (top) or, for those given in Table \ref{tab1}, $R > 1$ (bottom).  In the latter case, $R \approx 15$. Time is measured in days.}
\label{illustrative}
\end{figure}

Our analysis reveals one very important fact about the overall system: for starting values sufficiently close to equilibrium, the long term behavior depends only on the value of $R$. If $R > 1$ then the system tends towards an end state with a non-zero population of infected cells and virions (\textbf{viral persistence}), but if $R \leq 1$ then the final equilibrium is a state with no virus or infection (\textbf{viral extinction}). Figure \ref{illustrative} serves as an example that illustrates solutions in the two different cases given by Theorems \ref{PersistStability} and \ref{ExtinctStability}.  Finally, we also mention that global asymptotic stability of the equilibria can also be shown using a Lyapunov function as in \cite{Korob}.

Unfortunately, our analysis has been restricted to a deterministic model, and hence limited to only one possible solution, whereas in reality the systems under consideration may contain vast uncertainties, especially with regards to the parameters previously described.
In the next section, we incorporate chance mechanisms for population coefficients in order to estimate the probability that a persistent infection develops upon an initial viral load being transmitted to a new host.

%%%%%%%%%%%%%%%%%%%%%%%%%%%%%%%%%%

\section{Numerical Simulations and the Probability of Persistence}

Both mathematical and biological results support the idea that contact with HIV does not automatically imply the development of a persistent infection. Given factors such as the CD4+ T-cell growth rate, infection rate, and viral clearance rate, the theorems of the previous section display that it is possible to accurately predict the end viral state in the model. While this is very useful, it does not take into account the variability in parameter values amongst a group of individuals. To account for this, we now incorporate random coefficients for the early stages of HIV infection into the model and examine the resulting behavior.  
If these are introduced in a biologically meaningful fashion, we can estimate their contributions to the variability in the early time course of the viral load, which is not
possible with the deterministic coefficients of the previous section.  
Furthermore, we can obtain predictions of the probability that HIV levels reach certain values as a function of time since initial infection. 
Such levels can correspond to thresholds in various tests for the detection of HIV in blood. 

Another study \cite{TuckShip} previously estimated the probability of viral persistence using the three-component model with random variable coefficients by using the results of Theorems \ref{ExtinctStability} and \ref{PersistStability}. In this paper, the authors assumed that each random variable, except for $c$ which was assumed constant, possessed truncated normal distributions with the mean and standard deviation given by a clinical study of $10$ infected patients \cite{SCCDHP}.  Upon sampling from these distributions, the authors computed the value of $R$ and used this to determine the asymptotic behavior of the system as $t \to \infty$, thereby avoiding the need to directly simulate the model itself. The results of their simulations estimated that the average probability of viral extinction was approximately $1-7 \%$. 

There are a few issues with this approach that we plan to remedy in
the current section.  First and foremost, the validity of the model
declines rapidly after several months \cite{PerelsonSIAM, TuckShip} (approximately $100$ days).  Hence, determining viral
persistence or extinction based solely on the asymptotic behavior
of the system as $t \to \infty$ seems problematic.  One can easily find solutions which possess large
viral loads for several weeks, but eventually tend to extinction as $t
\to \infty$.  Another difficulty is that for certain coefficient
values, the virus population in the model may stay quite small for the
first few months after transmission, but then grow steadily to a
persistent steady state over large times.  These possible outcomes can
be seen more clearly within Figure \ref{fig1}.
Thus, instead of using the conditions $R \leq 1$ and $R > 1$ (which provide information
only about the behavior of (\ref{3CM}) in the limit as $t \to \infty$) to determine
whether or not a viral load has established a persistent infection, we will
formulate and utilize new conditions which possess a finite time horizon.
Since many standard tests for HIV currently display a threshold of detection 
of $50$ virions per $\mu$L \cite{RP, Rib}, we will set this as the barrier for viral persistence
at the end time of the model's validity.  Namely, the condition
$V(100) \geq 50$ will represent viral persistence, while $V(100) < 50$
will represent extinction.

\begin{center}
\begin{figure}[H]
  \begin{subfigure}[b]{0.5\textwidth}
\includegraphics[scale=.37]{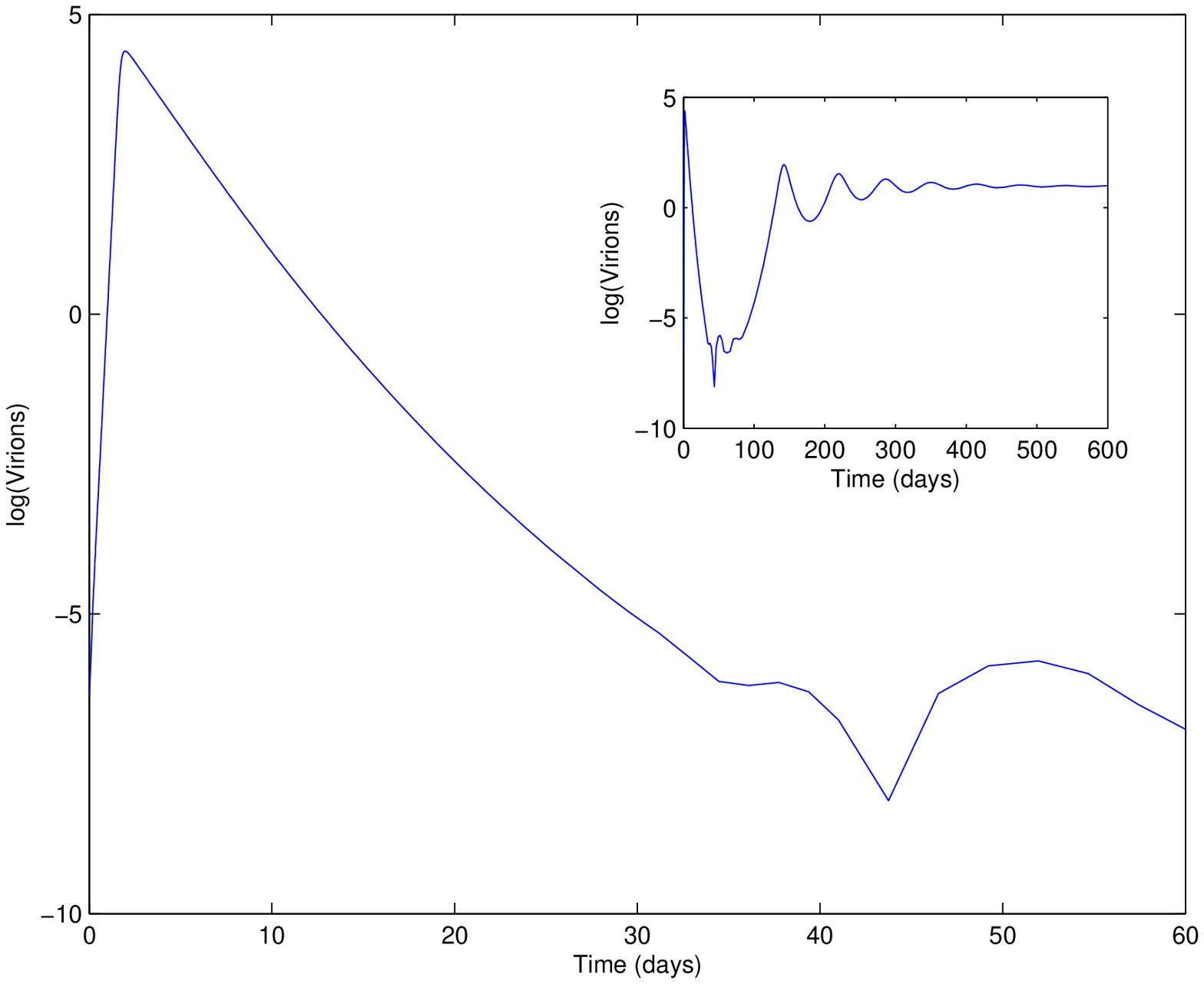}
  \end{subfigure}
\hspace{-0.35in}
\begin{subfigure}[b]{0.5\textwidth}
\includegraphics[scale=.37]{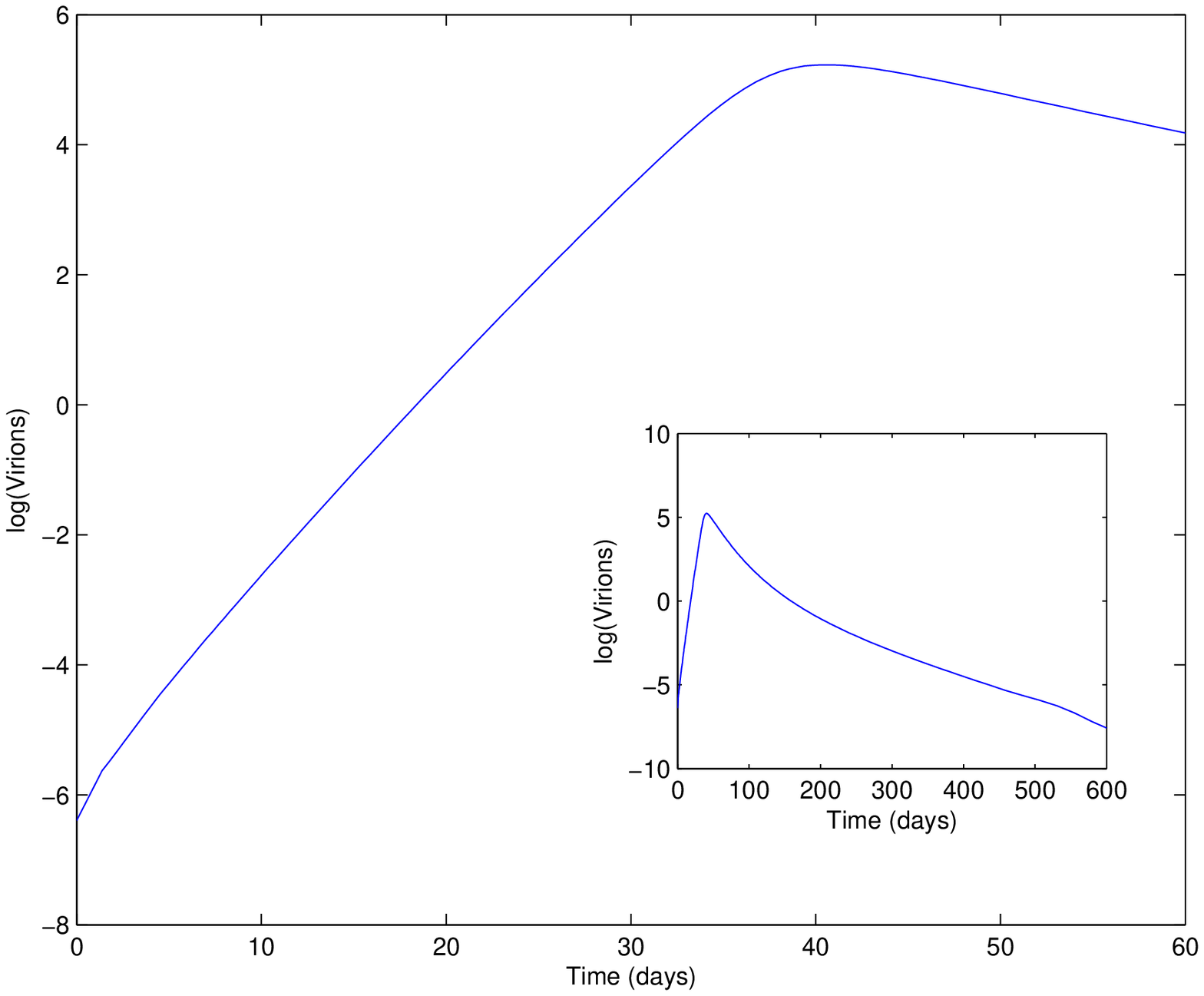}
\end{subfigure}
\caption{ \footnotesize A log plot of the virion population is shown
  (left) for a choice of parameters resulting in $R = 2.03$, yet the
  viral load remains small, around $10^{-7}$ copies per $\mu$ L, up to time $t=60$.  For
  larger times, however, the result of Theorem \ref{PersistStability}
  must apply and hence the viral population rebounds and settles into
  equilibrium near $10$ copies per $\mu$L by $t=600$ (left inset).  A log plot of the
  virion population is shown (right) for a choice of parameters
  resulting in $R = 0.74$, yet the viral load remains large, around
  $10^4$ copies per $\mu$L up to time $t=60$.  For larger times, however, the result of
  Theorem \ref{ExtinctStability} must apply and hence the viral
  population sharply decreases to zero, with values around $10^{-8}$ copies per $\mu$L
  by $t=600$ (right inset). }
\label{fig1}
\end{figure}
\end{center}

With these issues now in context, we perform a similar computational
study to estimate the probability of persistence. To simulate the
dynamics of the virus, we will employ a Monte-Carlo method along with
a traditional Runge-Kutta solver to compute solutions of the corresponding systems
of ordinary differential equations given by (\ref{3CM}).

\subsection{Sampling} \label{sec:sampling}
We are interested in examining multiple definitions of viral persistence
and allowing for a different parameter distribution than that of \cite{TuckShip}.
Hence, we consider two different cases.
In the first case, we sample from truncated normal distributions
as in \cite{TuckShip} to determine the values of the random variable coefficients $\lambda,
\mu, k, \delta$ and $p$, while keeping the parameter $c=3$ constant
throughout.
The probability of persistence is then estimated both using the time-asymptotic definition of
viral persistence (i.e. $R > 1$) and our new finite-time definition (i.e., $V(100) \geq 50$).
In the second case, we investigate the influence of the distribution of parameters by
sampling from uniform and triangular distributions, as previously performed for this system while considering
the additional effects of viral mutation \cite{RFP}.
Here, viral promoters (i.e., those parameters which lead to large reproductive numbers)
are sampled from uniform distributions, while viral inhibitors (in this case, death rates) are
sampled from triangular distributions, data for both of which are taken from Table \ref{tab1}. 
In particular, the promoter $k$ is sampled from a uniform distribution over the interval
$(1.9 \times 10^{-4}, 4.8 \times 10^{-3})$ $\mu$L/day similar to \cite{RFP}, and $p$ is sampled from another uniform distribution over
the interval $(98, 7100)$ day$^{-1}$.
As in \cite{BD}, we assume an asymmetric triangular distribution
$\mathrm{Tri}(0.0043, 0.01089, 0.02)$ day$^{-1}$ for $\mu$, the death rate of uninfected $T$-cells, where $0.0043$ is the
minimum value, $0.01089$ is the peak (occurring at the mean recorded value in Table \ref{tab1}) and $0.02$ is the maximum value.
The growth parameter $\lambda$ is set to be $10\mu$ as in \cite{TuckShip},
and the viral clearance rate $c$ is held constant at $3$ as in \cite{SCCDHP, TuckShip}. 
Finally, the death rate of infected $T$-cells, $\delta$ is sampled from another triangular distribution
$\mathrm{Tri}(0.13, 0.366, 0.8)$ day$^{-1}$.
Then, as before, the probability of persistence is estimated separately using the $R$ definition of
viral persistence and our finite-time horizon.

Of course, because our condition depends upon solving the ODEs until a specific time, 
our results may now have a strong dependence on initial conditions.  Thus, to generate
a suitable range of initial data, we estimate the probabilities of persistence and extinction
considering a variety of initial conditions.  In particular, we choose values of $T_0$ between
$100-1000$ cells/$\mu$L, initial viral loads $V_0$ between $100-500$ virions/$\mu$L, and
fix the initial infected $T$-cell population at $0$.  These values are obtained from similar initial
conditions of previous studies \cite{PerelsonSIAM, PKdB}.

\subsection{Results}

\begin{figure}[t]
\begin{center}
\includegraphics[scale=.6]{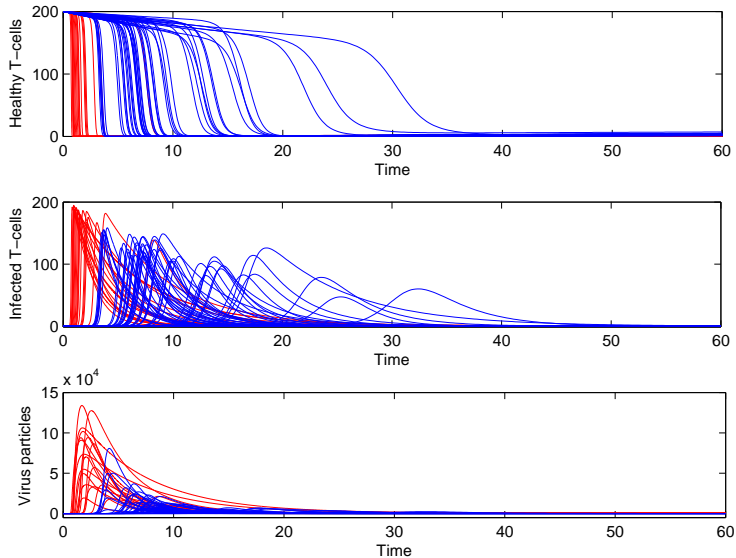}
\caption{\footnotesize Sample paths of the healthy T-cell, infected T-cell, and virion populations}
%\vspace{-1in}
\label{samplepaths}
\end{center}
\end{figure}

According to the described computational methods, $500,000$
simulations were conducted to determine the probability of persistence
of a specific infection.  Sample paths of associated trials are displayed in Figure \ref{samplepaths} 
and the results are summarized within Table \ref{table:rHorizons}.
Notice that the proportion of trials within our simulation that
resulted in $R \leq 1$ is much smaller - by a factor of ten or twenty -
than those which resulted in viral populations less than $50$ virions
per $\mu$L approximately three months after the initial infection.
Hence, a standard mathematical definition of extinction, namely
$\lim_{t \to \infty} V(t) = 0$ which (by Theorem
\ref{ExtinctStability}) is equivalent to $R \leq 1$, appears to be insufficient to
accurately describe the behavior of the viral population on the
timescales of biological relevance, specifically during the time
period up to a few months after initial infection.  Instead, a more
precise determination of viral extinction can be made by measuring
whether the virus population will remain below the current detectable
threshold of $50$ virions per $\mu$L, and under this measure, the
probability of extinction is much larger.  As can be seen by Table
\ref{table:rHorizons}, the probability of extinction by the former measure is
only around $1-2\%$, while the probability jumps to
approximately $8-15\%$ under the latter notion that we have proposed.
In some types of transmission, such as passage of the disease from a
mother to an unborn child or through needlesticks, this percentage is
much closer to current estimates of the probability of extinction
after initial infection than the criteria $R\leq 1$.

\begin{table}[t]
\centering
\begin{tabular}{|c|c|c|}
\hline 
Parameter Distribution & $\mathbb{P} \left (R \leq 1 \right )$  & $\mathbb{P} \left ( V(60)<50 \frac{\mbox{copies}}{\mu \mbox{L}} \right )$ \\ \hline  Truncated & 0.0136 & 0.1510 \\  Normal & & \\ \hline Uniform \& & 0.0046 & 0.0859 \\ 
Triangular & & \\ \hline
\end{tabular}

\vspace{.3cm}
\caption{ \footnotesize Probabilities of virion extinction, using finite-time and time-asymptotic
definitions of extinction, as well as different parameter distributions 
(truncated normal and uniform/triangular, respectively). A total of $500,000$ trials were performed
for each case; for the time-asymptotic probabilities, initial conditions were varied
over 50 combinations, with $V(0)$ varying from 100-500 $\mu L^{-1}$ and $T(0)$ varying from
100-1000 $\mu L^{-1}$, and $I(0)$ set at 0. Hence, for each initial condition pair $(T(0),V(0))$, $10,000$
trials were performed.}
\label{table:rHorizons}
\end{table}

%%%%%%% Begin Eric additions

Based on the results of our simulations, which have been consolidated in 
Table \ref{table:rHorizons}, the probability of virus extinction using
the finite time definition is significantly less than the
extinction probability of the time-asymptotic definition by more
than an order of magnitude. This result marks the distinction
between the finite- and infinite-time results, and also alludes
to the eventual breakdown of the model for $t > 100$ as predicted results for (\ref{3CM})
stray from clinical results \cite{PerelsonSIAM, TuckShip}. 

\begin{center}
\begin{figure}[t]
\includegraphics[width=\textwidth]{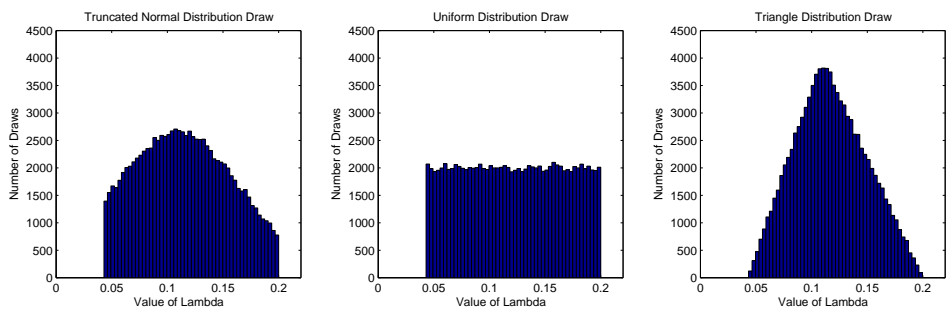}
\caption{ \footnotesize Average values of $\lambda$ over $100,000$ draws using a truncated normal distribution (left), a 
uniform distribution (center), and a triangular distribution (right).}
\label{randomDraws}
\end{figure}
\end{center}

While the probability of viral extinction associated with the finite-time 
definition remained ten times higher than the corresponding asymptotic-time
probability for each distribution, the extinction probabilities 
across the two distributions also varied by almost a factor of two. 
This  demonstrates the system's sensitivity to variations 
within the range of currently accepted values. 
In Figure \ref{randomDraws} the difference between parameter distributions is
shown, providing a description of how variations in parameter 
values can prompt a change in the model behavior. 

%However, the methods of drawing 
%values as well as the values used in determining the distributions 
%can explain this decrease in extinction probability: the variation 
%in R can be reduced to be proportional to $k$ and $p$, and 
%inversely proportional to $\delta$ (this reduction can be made 
%because $\lambda$ and $\mu$ are defined to be linearly associated 
%and because $c$ is fixed and hence has no variation). Then, since 
%the mean of $\delta$ is less than the uniform distribution 
%parameter values $\delta$ is being compared to,  $R$ will be 
%inflated relative to the truncated normal distribution, decreasing 
%the probability for viral extinction.

\begin{figure}[t]
\centering 
\includegraphics[scale = 0.37]{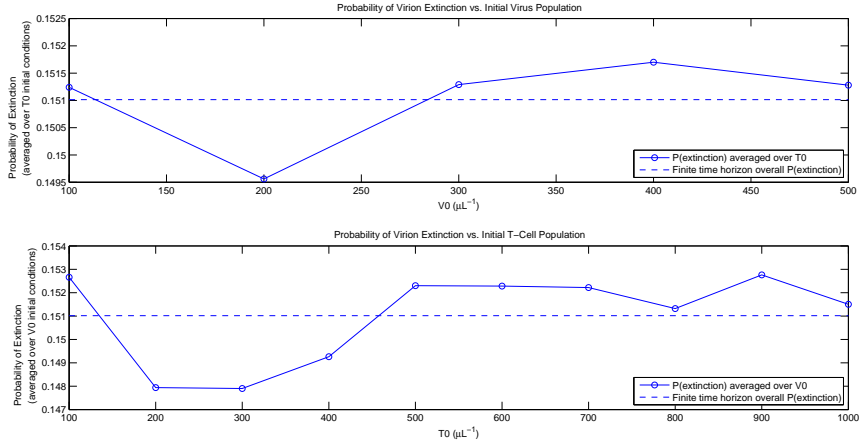}
\caption{ \footnotesize From parameters drawn from truncated normal distributions (as described in Section \ref{sec:sampling}), 
these graphs show probabilities of virus extinction over varying initial conditions. $V(0)$ is varied from 100 to 500 while 
$T(0)$ is held constant (top), and $T(0)$ is varied from 100-1000 while $V(0)$ is held constant (bottom). $I(0)$ is kept 
fixed at $0$ for all trials. Throughout all simulations, the variation in extinction probability remains small.}
\label{truncAvg}
\end{figure}

Lastly, it was expected 
that the initial conditions for T-cell and virus populations would 
have a significant impact on the probability of virus extinction, 
but based on our simulations this is not the case. As seen in 
Figures \ref{truncAvg} and \ref{triAvg}, for both distributions the 
initial conditions played little role: for each distribution, when 
averaged over the initial T-cell and initial virion conditions, the 
probability of persistence varied less than 5$\%$ from the mean 
value (averaged over all initial conditions). In addition, there is 
no clear pattern in the residuals of the probability of persistence 
of each initial condition relative to the overall probability, 
which suggests that the initial conditions have little to no effect 
to the disease's end behavior. This same phenomena occurs in the 
asymptotic-time horizon definition of persistence, where the 
initial conditions are irrelevant to the probability of persistence 
of the virus. 

\begin{figure}[t]
\centering 
\includegraphics[scale = 0.37]{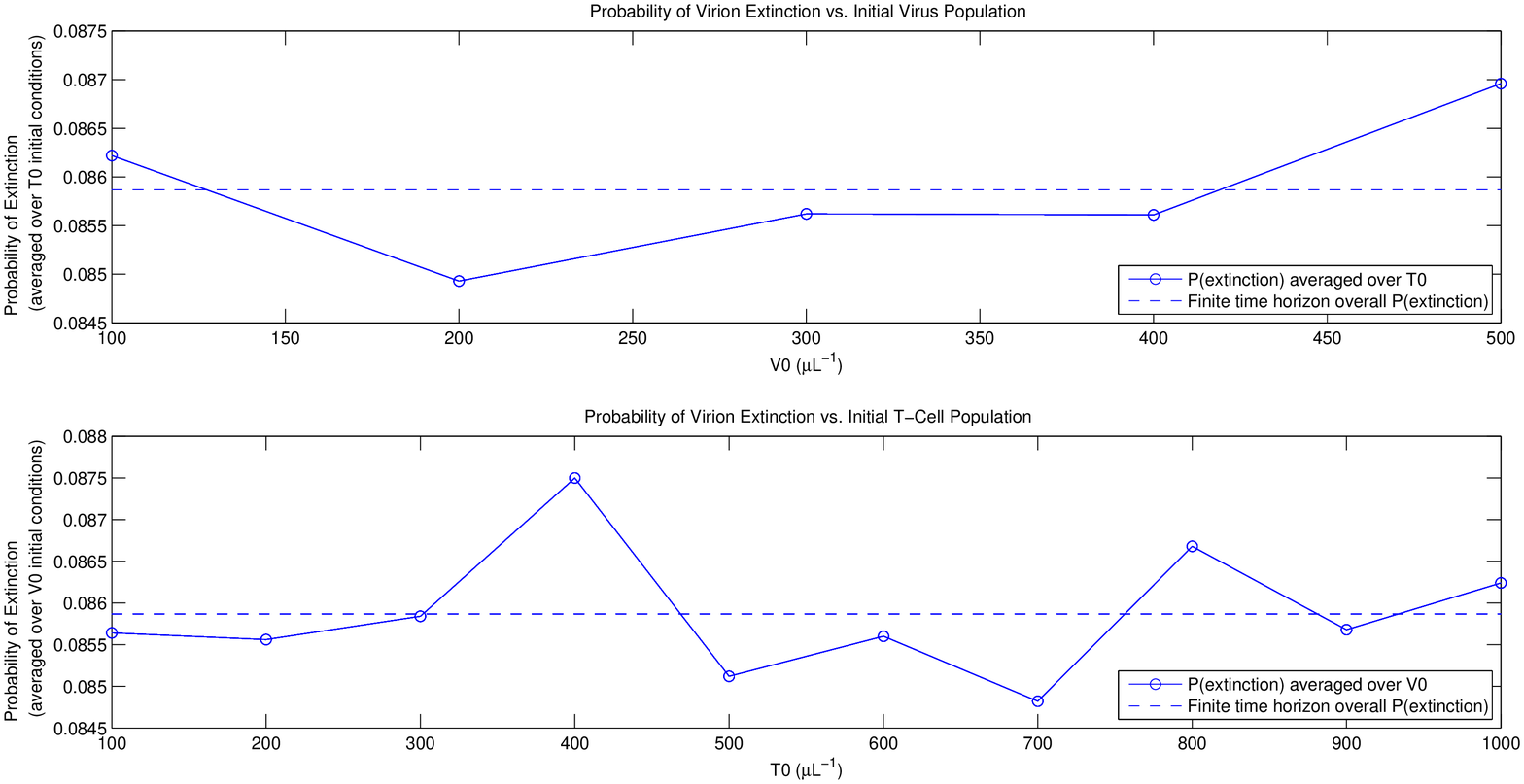}
\caption{\footnotesize From parameters sampled from uniform and triangular distributions (as described in Section 
\ref{sec:sampling}), these graphs show probabilities of virus extinction over varying initial conditions. $V(0)$ is varied from 
100 to 500 while $T(0)$ is held constant (top), and $T(0)$ is varied from 100-1000 while $V(0)$ is held constant (bottom). 
$I(0)$ is kept fixed at $0$ for all trials. }
\label{triAvg}
\end{figure}

%%%%%%%%%% End Eric additions

%\begin{figure}[H]
%\begin{center}
%\includegraphics[scale=.6]{Images/ODEHist.png}
%\caption{\footnotesize Histograms of end state values from ODE with
%  random variable coefficients}
%\label{hist1}
%\end{center}
%\end{figure}

%%%%%%%%%%%%%%%%%%%%%%%%%%%%%%%

\section{Conclusions}

Three main accomplishments are noted within the current article.
First, a local mathematical analysis was performed, and theorems
were proved to justify the viability and utility of the three-component
model.
Second, the local asymptotic stability of steady states were proved
and used to define mathematical notions for viral persistence and
extinction.  These conditions then inspire ODE-free simulations
of the in-host viral dynamics merely be sampling from distributions
describing the variation of parameters amongst a population of
exposed individuals.
Finally, we proposed alternative definitions for the notions of viral persistence
and extinction, which possess a finite time horizon, rather than
depending only upon the asymptotic limit as $t \to \infty$.  
Simulations were performed to measure the differences in these criteria
and their dependence on the probability distribution of parameters. 
We note that the end viral population, $V(100)$, could be less than the threshold of
detectability permitted by modern science much more often than
predicted by the associated value of the reproduction ration $R$. This indicates that the methods used 
in \cite{TuckShip} are not as accurate as could be hoped for.
Hence, it should be clear that a full simulation of the model is required to
obtain accurate results regarding the probability of developing a persistent infection. 

With full simulation of the ordinary differential equations, we
discovered that viral persistence occurred at a rate of roughly ten
times that suggested in \cite{TuckShip}. 
This proves rather conclusively the finite time and asymptotic limits conditions
while related, do not yield identical predictions, and simulation is necessary to
realize the full implications of the model. 
In reality, persistence has been estimated to occur at a rate of up to $90\%$, if the virus is transmitted by
blood transfusion, or around $1\%$ if transmitted via sexual intercourse
\cite{Gray}. While \cite{TuckShip} arrived at an average probability of
persistence around $93-99\%$ and we approximated this figure to be around $85-92\%$, recall that these coefficients were drawn from
a biased population of HIV-infected individuals because the data
arose from people known to have already developed
a persistent HIV infection.  Additionally, the previous estimates already account for the probability of transmission within 
them, while our estimates of persistence and those of \cite{TuckShip} specifically assume that transmission has occurred within a new host.

Stochastic models, such as in \cite{TuckLC}, that utilize Brownian
motion to incorporate additional random effects stemming from factors
outside of parameter estimation could provide additional improvement
to our estimates of persistence, and this could be the goal of a future
project.  Due to the lack of mathematical tools to analyze
stochastic differential equations, though, one would be forced to
resort to a more computational framework, rather than a mathematical
or asymptotic analysis, in order to glean information from the model.
With the development of more descriptive models and more advanced
analytical and computational tools, the probability of HIV persistence
after initial infection can be estimated with even greater precision.

\section{Acknowledgements} This work was completed under the direction
of Prof. Mrinal Raghupathi and Prof. Stephen Pankavich, and submitted
in partial fulfillment of MIDN Roemer's Trident Scholar thesis
requirement.  Additionally, the authors were supported in part by
National Science Foundation grants DMS-0908413 and DMS-1211667.

%\section{Appendix}
%\label{Appendix}

\bibliographystyle{siam}
%\bibliography{Bib}

\end{document}